\documentclass[11pt]{amsart}

\usepackage[utf8]{inputenc}
\usepackage[margin=1in]{geometry}
\usepackage[titletoc,title]{appendix}
\usepackage{pifont}
\usepackage{yfonts}
\RequirePackage[dvipsnames,usenames]{xcolor}
\usepackage[colorlinks=true,linkcolor=blue]{hyperref}%
\hypersetup{
    colorlinks=true,
    linkcolor=RoyalBlue,
    filecolor=magenta,      
    urlcolor=RoyalBlue,
    pdfpagemode=FullScreen,
    citecolor=BurntOrange
}

\usepackage{amsmath,amsfonts,amssymb,mathtools}
\usepackage{enumitem}
\usepackage{graphicx,float}

\usepackage[ruled,vlined]{algorithm2e}
\usepackage{algorithmic}
\usepackage{amsthm}

\usepackage{ytableau}
\usepackage{comment}
\usepackage{bigints}
\usepackage{setspace}
\numberwithin{equation}{section}

\newtheorem{defn}{Definition}[section]
\newtheorem{theorem}{Theorem}[section]
\newtheorem{corollary}[theorem]{Corollary}
\newtheorem{lemma}[theorem]{Lemma}
\newtheorem{prop}[theorem]{Proposition}
\newtheorem{question}[theorem]{Question}
\newtheorem{remark}[theorem]{Remark}
\newtheorem{example}[theorem]{Example}

\newcommand{\frakm}{\mathfrak{m}}
\newcommand{\frakj}{\mathfrak{J}}

\newcommand{\ZZ}{\mathbb{Z}}

\DeclareMathOperator{\Hom}{Hom}
\DeclareMathOperator{\Ext}{Ext}

\DeclareMathOperator{\Dim}{dim}

\DeclareMathOperator{\GL}{GL}
\DeclareMathOperator{\Sym}{Sym}
\DeclareMathOperator{\Pf}{Pf}

\makeatletter
\@namedef{subjclassname@2020}{%
  \textup{2020} Mathematics Subject Classification}
\makeatother

\title{On the Generalized Multiplicities of Maximal Minors and Sub-Maximal Pfaffians}
\author{Jiamin Li}
\address{Department of Mathematics, Statistics, and Computer Science, University of Illinois at Chicago,
Chicago, IL 60607}
\email{jli283@uic.edu}
\subjclass[2020]{}

\begin{document}

\maketitle

\begin{abstract}
	Let $S=\mathbb{C}[x_{ij}]$ be a polynomial ring of $m\times n$ generic variables (resp. a polynomial ring of $(2n+1) \times (2n+1)$ skew-symmetric variables) over $\mathbb{C}$ and let $I$ (resp. $\Pf$) be the determinantal ideal of maximal minors (resp. sub-maximal pfaffians) of $S$. Using the representation theoretic techniques introduced in the work of Raicu et al, we study the asymptotic behavior of the length of the local cohomology module of determinantal and pfaffian thickenings for suitable choices of cohomological degrees. This asymptotic behavior is also defined as a notion of multiplicty. We show that the multiplicities in our setting coincide with the degrees of Grassmannian and Orthogonal Grassmannian.
\end{abstract}

\section{Introduction}
Let $S=\mathbb{C}[x_{ij}]_{m\times n}$ be a polynomial ring of $m \times n$ variables with $m \geq n$. When $[x_{ij}]_{m \times n}$ is a generic matrix and $m>n$, we denote the determinantal ideals generated by $p \times p$ minors by $I_p$. On the other hand if $[x_{ij}]_{m \times n}$ is a skew-symmetric matrix with $m=n$ then we denote the ideals generated by its $2p \times 2p$ pfaffians by $\Pf_{2p}$. Our goal in this paper is to study the generalized multiplicities of $I_n$ and $\Pf_{2n}$, which is also a study of asymtoptic behavior of the length of the local cohomology modules. The precise definition of the generalized multiplicities will be given later. Our main theorems are the following.

\begin{theorem}(Theorem \ref{formula_limit}\label{main_thm})
	Let $S = \mathbb{C}[x_{ij}]_{m\times n}$ where $m>n$ and $[x_{ij}]_{m\times n}$ is a generic matrix, then we have
	\begin{enumerate}
		\item If $j\neq n^2-1$, then $\ell(H^j_\frakm(S/I_n^D))$ and $\ell(H^j_\frakm(I_n^{d-1}/I_n^d))$ are either $0$ or $\infty$. 
		\item If $j=n^2-1$, then $\ell(H^j_\frakm(S/I_n^D))$ and $\ell(H^j_\frakm(I^{d-1}_n/I^d_n))$ are nonzero and finite. Moreover we have 
			\begin{align}\label{main_thm_formula}
				\begin{split}
				&\lim_{d \rightarrow \infty}\dfrac{(mn-1)!\ell(H^j_\frakm(I_n^{d-1}/I_n^d))}{d^{mn-1}} \\&= (mn)!\prod^{n-1}_{i=0}\dfrac{i!}{(m+i)!}
				\end{split}
			\end{align}
		\item In fact the limit $$\lim_{D\rightarrow \infty} \dfrac{(mn)!\ell(H^j_\frakm(S/I_n^D))}{D^{mn}}$$ is equal to (\ref{main_thm_formula}) as well.
	\end{enumerate}
\end{theorem}
Suprisingly, $(\ref{main_thm_formula})$ is in fact the degree of the Grassmannian $G(n,m+n)$, see Remark \ref{Grassmannian}. A fortiori, it must be an integer. Moreover, (\ref{main_thm_formula}) can be interpreted as the number of fillings of the $m \times n$ Young diagram with integers $1,...,mn$ and with strictly increasing rows and columns, see \cite[Ex 4.38]{EisenbudHarris}.

Analogously, we have
\begin{theorem}\label{main_thm2}(Theorem \ref{formula_limit_pfaffian})
	Let $S = \mathbb{C}[x_{ij}]_{(2n+1)\times (2n+1)}$ where $[x_{ij}]_{(2n+1)\times (2n+1)}$ is a skew-symmetric matrix, then we have
	\begin{enumerate}
		\item If $j \neq 2n^2-n-1$, then $\ell(H^j_\frakm(S/\Pf_{2n}^D))$ and $\ell(H^j_\frakm(\Pf_{2n}^{d-1}/\Pf_{2n}^d))$ are either $0$ or $\infty$.
		\item If $j=2n^2-n-1$, then $\ell(H^j_\frakm(S/\Pf_{2n}^D))$ and $\ell(H^j_\frakm(\Pf_{2n}^{d-1}/\Pf_{2n}^d))$ are finite and nonzero. Moreover we have, 
			\begin{align}\label{main_thm2_formula}
				\begin{split}
		&\lim_{D \rightarrow \infty}\dfrac{(2n^2+n-1)!\ell(H^j_\frakm(\Pf_{2n}^{d-1}/\Pf_{2n}^d))}{d^{(2n^2+n-1)}} \\&= (2n^2+n)!\prod^{n-1}_{i=0}\dfrac{(2i)!}{(2n+1+2i)!}
				\end{split}
			\end{align}
			\item In fact the limit $$\lim_{D\rightarrow \infty} \dfrac{(2n^2+n)!\ell(H^j_\frakm(S/\Pf_{2n}^D))}{D^{2n^2+n}}$$ is equal to (\ref{main_thm2_formula}) as well.
	\end{enumerate}
\end{theorem}

Similar to Theorem \ref{main_thm}, $(\ref{main_thm2_formula})$ has a geometric interpretation, and it is the degree of the Orthogonal Grassmannian $OG(2n,4n+1)$, see Remark \ref{Orthogonal_Grass}. Therefore it must be an integer as well. This explains the similarities between the Hilbert-Samuel multiplicity and the multiplicities we discuss above. Furthermore, as in the case of Grassmannian, (\ref{main_thm2_formula}) can be interpreted similarly using the shifted standard tableaux, see \cite[p91]{Totaro} for the discussion.

As mentioned before, the above limits are notions of muliplicity. The Hilbert-Samuel multiplicity (see \cite[Ch4]{BrunsHerzog} for more detailed discussion) denoted by $e(I)$, has played an important role in the study of commutative algebra and algebraic geometry. The attempt of its generalization can be traced back to the work of Buchsbaum and Rim \cite{BuchsbaumRim} in 1964. One of the more recent generalizations is defined via the $0$-th local cohomology (see for example \cite{CutkoskyHST},\cite{KatzValidashti}, \cite{UlrichValidashti}), which coincides with the Hilbert-Samuel multiplicity when the ideal is $\mathfrak{m}$-primary. In \cite{CutkoskyHST}, the authors proved the existence of the $0$-multiplicities when the ring is a polynomial ring. Later, Cutkosky showed in \cite{Cutkosky} that the $0$-multiplicities exists under mild assumption of the ring. 

In \cite{JeffriesMontanoVarbaro} the authors studied this $0$-multiplicities of several classical varieties, in particular they calculated the formula of the $0$-multiplicities of determinantal ideals of non-maximal minors and the pfaffians. A further generalized multiplicity is defined in \cite{DaoMontano19} via the local cohomology of arbitrary indices, which is necessary in some situations, e.g. the determinantal ideals of maximal minors. However, the existence of such multiplicity is not known in general. In the unpublised work \cite{Kenkel} the author calculated the closed formula, and thus showed the existence, of the generalized $j$-multiplicity defined in \cite{DaoMontano19} of determinantal ideals of maximal minors of $m\times 2$ matrices. Thus our Theorem \ref{main_thm} and Theorem $\ref{main_thm2}$ are extensions of the results in \cite{JeffriesMontanoVarbaro} and \cite{Kenkel}.

We give the definiton of generalized multiplicities here.

\begin{defn}\label{defn_dm}(see \cite{DaoMontano19} for more details)
	Let $S$ be a Noetherian ring of dimension $k$ and $\frakm$ a maximal ideal of $S$. Let $I$ be an ideal of $S$. Define
$$\epsilon_+^j(I) := \limsup_{D\rightarrow\infty}\dfrac{k!\ell(H^j_\frakm(S/I^D))}{D^k}.$$
	Suppose $\ell(H^j_\frakm(S/I^D)) < \infty$, then we define
	$$\epsilon^j(I) := \lim_{D \rightarrow \infty} \dfrac{k!\ell(H^j_\frakm(S/I^D))}{D^{k}}$$ if the limit exist and we call it the $j$-$\epsilon$-multiplicity. 
\end{defn}

\begin{defn}\label{defn_higher_mul}
	Under the same setting, we define $$\mathfrak{J}^j_+(I) := \limsup_{d\rightarrow \infty} \dfrac{(k-1)!\ell(H^j_\frakm(I^{d-1}/I^d))}{d^{k-1}}.$$
	If $\ell(H^j_\frakm(I^{d-1}/I^d)) < \infty$, then we define $$\mathfrak{J}^j(I) := \lim_{D\rightarrow \infty} \dfrac{(k-1)!\ell(H^j_\frakm(I^{d-1}/I^d))}{d^{k-1}}.$$ if the limit exists and we call it the $j$-multiplicity.
\end{defn}

When $I$ is a $\frakm$-primary ideal, we have $$e(I) = (\Dim(S))!\lim_{t\rightarrow \infty} \dfrac{\ell(S/I^t)}{t^{\Dim(S)}} = (\Dim(S)-1)!\lim_{t\rightarrow \infty} \dfrac{\ell(I^{t-1}/I^t)}{t^{\Dim(S)-1}}.$$ However in general we may have $\epsilon^j(I) \neq \frakj^j(I)$, as we can see in the below results of $\epsilon^0(I_p)$ and $\Pf_{2p}$ for $p < n$ in \cite{JeffriesMontanoVarbaro}.

\begin{theorem}(See \cite[Theorem 6.1]{JeffriesMontanoVarbaro})\label{0_multiplicities_result}
	Let $I_p$ be the determinantal ideal of $p \times p$-minors of $S$ where $S$ is a polynomial ring of generic $m \times n$ variables over $\mathbb{C}$ and $0<p<n\leq m$. Let 
$$c=\dfrac{(mn-1)!}{(n-1)!...(n-m)!m!(m-1)!...1!},$$
then we have
\begin{enumerate}
\item $$\epsilon^0(I_p) = cmn\int_{\Delta_1}(z_1...z_n)^{m-n}\prod_{1\leq i < j\leq n}(z_j-z_i)^2dz$$ where $\Delta_1=\operatorname{max}_i\{{z_i}\}+t-1\leq \sum z_i \leq t\} \subseteq [0,1]^n$, 
\item $$\frakj^0(I_p) = cp\int_{\Delta_2}(z_1...z_n)^{m-n}\prod_{1\leq i<j \leq n}(z_j-z_i)^2 dz$$ where $\Delta_2  = \{\sum z_i = t\} \subseteq [0,1]^n$.
\end{enumerate}
\end{theorem}

The authors have also proved a corresponding theorem for the skew-symmetric matrix.

\begin{theorem}(See \cite[Theorem 6.3]{JeffriesMontanoVarbaro})
	Let Let $\Pf_{2p}$ be the $2p \times 2p$ pfaffians of a polynomial ring $S$ with $n \times n$ skew-symmetric variables. Let $m := \lfloor n/2 \rfloor$. Then for $0 < p < m$, let $$c = \dfrac{(\binom{n}{2}-1)!}{m!(n-1)!...1!},$$ we have
	\begin{enumerate}
		\item	$$\epsilon^0(\Pf_{2p}) = c\binom{n}{2}\int_{\Delta_1}(z_1...z_m)^{2y}\prod_{1\leq i < j \leq m}(z_j-z_i)^4dz$$
		\item $$\frakj^0(\Pf_{2p}) = cp\int_{\Delta_2}(z_1...z_n)^{2y}\prod_{1\leq i < j\leq m}(z_j-z_i)^4 dz.$$
	\end{enumerate}
			where $y=0$ if $n$ is even and $1$ otherwise, and $\Delta_1$ and $\Delta_2$ are the same as those in Theorem \ref{0_multiplicities_result}.
\end{theorem}
Note that when $S$ is a polynomial ring of $m\times n$ generic variables and $p=n$, $H^0_\frakm(S/I_n^D)$ is always $0$, respectively when $S$ is a polynomial ring of $(2n+1) \times (2n+1)$ skew-symmetric variables and $q=n$, we have $H^0_\frakm(S/\Pf_{2n}^D) = 0$. To avoid this triviality we will instead study the multiplicites of $I_n$ and $\Pf_{2n}$ of higher cohomological indices, which will require more tools from representation theory.

It was proved in \cite{DaoMontano19} that when $S$ is a polynomial ring of $k$ variables and when $J$ is a homogeneous ideal of $S$, we have for all $\alpha \in \mathbb{Z}$, 
$$\limsup_{D\rightarrow\infty}\dfrac{k!\ell(H^j_\frakm(S/I^D)_{\geq \alpha D})}{D^k} < \infty.$$
As a corollary of the above result, combined with the result from \cite{Raicu}, which states that if $S$ is a polynomial ring of $m \times n$ variables and $I_p$ is a determinantal ideal of $p \times p$-minors of $S$, then $H^i_\frakm(S/I_p^D)_j = 0$ for $i\leq m+n-2$ and $j<0$, we get that $\epsilon_+^j(I_p) < \infty$ for $j\leq m+n-2$ (see \cite[Ch 5]{DaoMontano19}). Note that, as mentioned in \cite[Ch 7]{DaoMontano19}, even if $\epsilon^j(I)$ exists, it doesn't have to be rational (see the example in \cite[Ch 3]{CutkoskyHST}). Therefore it is natural to ask for which $j$ the multiplicities exist, and if they exist, the rationality of the multiplicities. As we see in Theorem \ref{main_thm} and Theorem \ref{main_thm2}, the only interesting cohomological indices to our question are $n^2-1$ for maximal minors and $2n^2-n-1$ for sub-maximal pfaffians, and we solve the problem of calculating the generalized multiplicites of determinantal ideals of maximal minors and sub-maximal pfaffians completely.




\textbf{Organization.} In section 2 we will recall briefly the construction of Schur functors. In section 3 we will review the $\Ext$-module decompositions in the case of determinantal thickenings of generic matrix and derive some useful properties. Then we will show the existence calculate the $j$-multiplicity in section 4. We will follow the same strategies for skew-symmetric matrix in section 5 and 6. Finally, we will discuss some future directions of this line of work in section 7.

\medskip
\textbf{Notations.}
In this paper $\ell(M)$ will denote the length of a module $M$, $S$ will denote the polynomial ring $\mathbb{C}[x_{ij}]$. We will use $D$ to denote the powers of ideals when we discuss modules related to $\epsilon^j(I)$ and use $d$ to denote the powers of ideal when we dicuss modules related $\frakj^j(I)$. All rings are assumed to be unital commutative.

\section{Preliminaries on Schur Functor}
We will recall the basic construction of the Schur functors, more information can be found in \cite{FultonHarris} and \cite{Weyman}. Let $V$ be an $n$-dimensional vector space over $\mathbb{C}$. Denote the collection of partitions with $n$ nonzero parts by $\mathcal{P}(n)$. We define a dominant weight of $V$ to be $\lambda = (\lambda_1,...,\lambda_n) \in \mathbb{Z}^n$ such that $\lambda_1\geq ... \geq \lambda_n$ and denote the set of dominant weights to be $\ZZ_{\text{dom}}$. Note that $(\lambda_1,\lambda_2,0,0,...,0) = (\lambda_1,\lambda_2)$. Furthermore we denote $(c,...,c)$ by $(c^n)$. We say $\lambda=(\lambda_1,\lambda_2,...) \geq \alpha = (\alpha_1,\alpha_2,...)$ if each $\lambda_i \geq \alpha_i$. Given a weight we can define an associated Young diagram with numbers filled in. For example if $\lambda=(3,2,1)= (3,2,1,0,0,0)\in \mathbb{Z}^6$, then we can draw the Young diagram 
\[
\begin{ytableau}
	1 & 2 & 3\\
	4 & 5\\
	6\\	
\end{ytableau}
\]
Let $\mathfrak{S}_n$ be the permutation group of $n$ elements. Let $P_\lambda = \{g\in\mathfrak{S}_n:g \text{ preserves each row}\}$ and $Q_\lambda=\{g\in\mathfrak{S}_n:g \text{ preserves each column}\}$. Then we define $a_\lambda = \sum_{g\in P_\lambda}e_g$, $b_\lambda = \sum_{g\in Q_\lambda}\operatorname{sgn}(g)e_g$, and moreover $c_\lambda  = a_\lambda \cdot b_\lambda$. 

Recall that the Schur functor $S_\lambda(-)$ is defined to $$S_\lambda(V) = \operatorname{Im}(c_\lambda\big|_{V^{\otimes \mu}})$$ where $\mu = |\lambda|$.

Let $V$ be an $n$-dimensional $\mathbb{C}$-vector space. We have a formula for the dimension of $S_\lambda V$ as $\mathbb{C}$-vector space.
\begin{prop}(See \cite[Ch2]{FultonHarris})\label{dim_schur}
	Suppose $\lambda = (\lambda_1,...,\lambda_n) \in \mathbb{Z}^n_\text{dom}$. Then we have $$\Dim(S_\lambda V) = \prod_{1\leq i < j \leq n}\dfrac{\lambda_i-\lambda_j+j-i}{j-i}.$$
\end{prop}
From the formula of $\Dim(S_\lambda V)$ it is easy to see the following.
\begin{corollary}\label{same_dim}
	For any $c\in \mathbb{N}$ we have
	$$\Dim(S_\lambda V) = \Dim(S_{\lambda+(c^n)}V).$$
\end{corollary}
\section{Decompositions of Ext modules of determinantal thickenings of maximal minors}
In this section we recall the $\operatorname{GL}$-equivariant $\mathbb{C}$-vector spaces decompositions of $\Ext^j_S(S/I_p^D)$ given in \cite{Raicu}. This will be the key ingredient in the disuccsion of multiplicities in section 4. 

Following the notations in \cite{Raicu}, we denote 
$$\mathcal{X}^d_p = \{\underline{x} \in \mathcal{P}(n) : |\underline{x}| = pd, x_1 \leq d\}.$$

Recall the following construction of finite set. First we define $x_i'$ to be the number of boxes in the $i$-th column of the Young diagram defined by $\underline{x}$. Then we define $\underline{x}(c)$ to be such that $\underline{x}(c)_i = \operatorname{min}(x_i,c)$.

\begin{defn}\label{Z_set}(See \cite[Definition 3.1]{Raicu})
	Suppose $\mathcal{X} \subset \mathcal{P}(n)$ is a finite subset. Then we define the set $\mathcal{Z}(\mathcal{X})$ to be the set consisted of the pair $(\underline{z},l)$ with $\underline{z} \in \mathcal{P}$ and $l \geq 0$. Let $z_1 = c$. Then we have 
	\begin{enumerate}
		\item There exists a partition $\underline{x} \in \mathcal{X}$ such that $\underline{x}(c) \leq \underline{z}$ and $x'_{c+1} \leq l+1$.
		\item If $\underline{x} \in \mathcal{X}$ satisfies (1) then $x'_{c+1} = l+1$.
	\end{enumerate}
\end{defn}

\begin{lemma}\label{z(x)_set}(See \cite[Lemma 5.3]{Raicu})
	Denote $\mathcal{Z}(\mathcal{X}^d_p)$ by $\mathcal{Z}^d_p$, then we have
\begin{align*}
    \mathcal{Z}^d_p=\big\{(\underline{z},l): 0\leq l \leq p-1, \underline{z}\in \mathcal{P}(n), z_1=...=z_{l+1}\leq d-1&, \\|\underline{z}| + (d-z_1)\cdot l +1 \leq p\cdot d \leq |\underline{z}|+(d-z_1)\cdot(l+1)\big\}.
\end{align*}
\end{lemma}
Next we recall the construction of the quotient $J_{\underline{z},l}$ from \cite{RaicuWeyman}, and will be crucial in the decomposition of the $\Ext$ modules of $\GL$-equivariant ideals. Let $\underline{z} = (z_1,...,z_m) \in \mathcal{P}(m)$ be such that $z_1 = ... = z_{l+1}$ for some $0 \leq l \leq m-1$. Then we define $$\mathfrak{succ}(\underline{z}, l) = \{\underline{y} \in \mathcal{P}(m) | \underline{y} \geq \underline{z} \text{ and } y_i > z_i \text{ for some } i>l \},$$ it is easy to see that $I_{\mathfrak{succ}(\underline{z},l)} \subseteq I_{\underline{z}}$, so we can define the quotient $J_{\underline{z},l} = I_{\underline{z}} / I_{\mathfrak{succ}(\underline{z},l)}$.

The above definition and lemma will be used again later when we study the case of pfaffians of skew-symmetric matrix. In section 3 and section 4 we consider $S=\mathbb{C}[x_{ij}]$ where $[x_{ij}]$ is a generic matrix of $m \times n$ variables. Recall that we have the $\operatorname{GL}$-equivariant decomposition (Cauchy formula) of $S$:$$S=\bigoplus_{\lambda\in \mathbb{Z}^{\text{dom}}_{\geq 0}} S_\lambda\mathbb{C}^m \otimes S_\lambda\mathbb{C}^n.$$

Denote by $I_\lambda$ to be the ideal generated by $S_\lambda\mathbb{C}^m \otimes S_\lambda\mathbb{C}^n$. It was shown in \cite{DeConciniEisenbudProcesi} that a $\operatorname{GL}$-equivariant ideal $I$ of $S$ can be written as $$I_\lambda = \bigoplus_{\mu\geq \lambda}S_\mu\mathbb{C}^m \otimes S_\mu\mathbb{C}^n,$$ and in particular the ideal of $p\times p$ minors is equal to $I_{(1^p)}$, moreover we have $I^d_p = I_{\mathcal{X}^d_p}$. Moreover we get that the $\GL$-invariant ideals are of the form $$I_\mathcal{X} = \bigoplus_{\lambda \in \mathcal{X}} I_\lambda$$ for $\mathcal{X} \subset \mathcal{P}(n)$. 

The following is the key tool of this paper. Note that in \cite{Raicu} the author considered the decomposition of $\Ext^j_S(S/I_\mathcal{X},S)$ in general, but here we only consider specifically the determinantal ideals.

\begin{theorem}\label{schur-decompos}(See \cite[Theorem 3.3]{RaicuWeyman}, \cite[Theorem 2.5, Theorem 3.2]{Raicu})
	There exists a $GL$-equivariant filtration of $S/I_p^d$ with factors $J_{\underline{z},l}$ which are quotients of $I_{\underline{z}}$. Therefore we have the following vector spaces decomposition of $\Ext^j_S(S/I^d_p,S)$:
	\begin{align}\label{decompose_first}
		\Ext_S^j(S/I_p^d,S) = \bigoplus_{(\underline{z},l) \in \mathcal{Z}^d_p} \Ext_S^j(J_{\underline{z},l},S)
	\end{align}
and we have
\begin{align}
\operatorname{Ext}_S^j(J_{(\underline{z},l)},S) = \bigoplus_{\substack{0\leq s \leq t_1 \leq ... \leq t_{n-l}\leq l\\ mn -l^2 -s(m-n)-2(\sum^{n-l}_{i=1}t_i)=j \\ \lambda \in W(\underline{z},l;\underline{t},s)}} S_{\lambda(s)}\mathbb{C}^m \otimes S_\lambda\mathbb{C}^n
\end{align}
where $\mathcal{P}_n$ is the collection of partitions with at most $n$ nonzero parts, which means $z_1\geq z_2 \geq ... \geq z_n \geq 0$. Moreover the set $W(\underline{z},l,\underline{t},s)$ consists of dominant weights satisfy the following conditions:
\begin{align}\label{res_weight_gen}
\begin{cases}
\lambda_n \geq l -z_l -m, \\
\lambda_{t_i+i} = t_i-z_{n+1-i}-m & i=1,...,n-l,\\
\lambda_s \geq s-n \text{ and } \lambda_{s+1} \leq s-m.
\end{cases}
\end{align}
and the $\lambda(s)$ is given by $$\lambda(s)=(\lambda_1,...,\lambda_s,(s-n)^{m-n},\lambda_{s+1}+(m-n),...,\lambda_n+(m-n)) \in \mathbb{Z}^m_{\text{dom}}.$$
	In fact in our case we have $\lambda_n=l-z_l-m$. This also implies that $t_{n-l}=l$.
\end{theorem}

In the rest of the paper we will assume $p=n$, i.e. we only focus on the maximal minors case.

\begin{lemma}\label{weights_max_minors}
	In Theorem \ref{schur-decompos} we have $l=n-1$. Therefore the pair $(\underline{z},l)$ in Theorem \ref{schur-decompos} is of the form $((c)^n,n-1)$ for $c\leq d-1$. In particular we have $((d-1)^n,n-1)$ in $\mathcal{Z}^d_n$.
\end{lemma}
\begin{proof}: Note the restriction $l\leq p-1$ gives $l\leq n-1$. It is easy to check that $((d-1)^n,n-1)$ is in $\mathcal{Z}^d_n$. On the other hand, assume that there exists $(\underline{z},l)$ in $\mathcal{Z}^d_n$ such that $l\leq n-2$. From Theorem \ref{schur-decompos} we have the restriction $$|\underline{z}|+(d-z_1)\cdot(l+1) \geq nd$$ when $p=n$. However by our assumption we have
\begin{align*}
    |\underline{z}|+(d-z_1)(l+1) &\leq |\underline{z}|+(d-z_1)(n-1)  \\&=  |\underline{z}|+d(n-1)-z_1(n-1) \\&\leq  nz_1+d(n-1)-z_1(n-1) \\&=   z_1+d(n-1) \\&\leq  d-1 +d(n-1) = nd - 1 < nd.
\end{align*}
Contradicting to our restriction. Therefore we must have $l=n-1$. Moreover, by the definition of $(\underline{z},l)$ we have $z_1=...=z_{l+1}$, therefore in our case we have $z_1=...=z_n$. So the $(\underline{z},l)$ is of the form $((c)^n,n-1)$ for $c\leq d-1$.
\end{proof} 
For the rest of section 3 and section 4 we will denote $I:=I_n$.
Using this information we can also gives a criterion of the vanishing of the $\Ext$ modules. Recall that the highest non-vanishing cohomological degree of $S/I^d_n$ is $n(m-n)+1$ (see \cite{Huneke}). This can be seen from the following lemma as well.
\begin{lemma}\label{vanishing_degree}
	In our setting $\Ext^j_S(S/I^d,S)\neq 0$ if and only if $m-n$ divides $1-j$ and $j\geq 2$. Moreover, $\Ext^{n(m-n)+1}_S(S/I^d,S)\neq 0$ if and only if $d\geq n$.
\end{lemma}
\begin{proof}
	By Lemma \ref{weights_max_minors}, the weights $\lambda\in W:=W(\underline{z},n-1,(n-1),s)$ have the restrictions
	\begin{align}\label{res_weight_max}
		\begin{cases}
		\lambda_n=n-1-z_{n-1}-m,\\
		\lambda_s\geq s-n \text{ and } \lambda_{s+1}\leq s-m.
		\end{cases}
	\end{align}
	We also have 
	\begin{align}\label{important_eq}
		mn-(n-1)^2-s(m-n)-2(n-1)=j \implies s(m-n)=n(m-n)+1-j.
	\end{align}
	By Theorem \ref{schur-decompos}, $\Ext^j_S(S/I^d_n,S)\neq 0$ if and only if the set $W$ is not empty, then by (\ref{res_weight_max}) and (\ref{important_eq}) this means $m-n$ divides $n(m-n)+1-j \implies m-n$ divides $1-j$ and $s=(n(m-n)+1-j)/(m-n) \leq l = n-1 \implies j\geq 2$. This proves the first statement.

	To see the second statement, note that when $j=n(m-n)+1 = mn-(n-1)^2-2(n-1)$ we have $s=0$. In this case we have the restriction
	\begin{align*}
		\begin{cases}
			\lambda_n = n-1-z_{n-1}-m,\\
			\lambda_1 \leq -m.
		\end{cases}
	\end{align*}
	If $d<n$, then $\lambda_n \geq n-d-m > -m \geq \lambda_1$, a contradiction, so that means the set $W$ is empty. On the other hand if $d\geq n$ then $W$ is not empty. So $\Ext^{n(m-n)+1}_S(S/I^d,S) \neq 0$ if and only if $d\geq n$. 
\end{proof}

In our proof of the main theorem, we will need an important property of the $\Ext$-modules, which only holds for maximal minors.

\begin{prop}\label{isom_ext_modules}(See \cite[Corollary 4.4]{RaicuWeymanWitt})
	We have $\Hom_S(I^d,S)=S$, $\Ext^1(S/I^d,S)=0$ and $\Ext^{j+1}_S(S/I^d,S) = \Ext^j_S(I^d,S)$ for $j>0$.
\end{prop}

\begin{lemma}(See \cite[Theorem 4.5]{RaicuWeymanWitt})\label{injectivity}
	Given the short exact sequence 
	$$0\rightarrow I^{d}\ \rightarrow I^{d-1} \rightarrow I^{d-1}/I^d \rightarrow 0,$$ 
	the induced map $$\Ext^j_S(I^{d-1},S)\hookrightarrow \Ext^j_S(I^{d},S)$$ is injective for any $j$ such that $\Ext^j_S(I^d,S)\neq 0$.
\end{lemma}

In order to prove our main theorem, we need to investigate the length of the $\Ext$-modules. We will need the following fact.

\begin{lemma}\label{length=dim}
	Given a graded $S$-module $M$ we have $\ell(M) = \Dim_\mathbb{C}(M)$.
\end{lemma}
\begin{proof}
	First assume $M$ is finitely graded over $\mathbb{C}$ and write $M=\oplus_i^\alpha M_i$. We will use the $\mathbb{C}$-vector space basis of each $M_i$ to construct the composition series of $M$ over $S$. Suppose $M_\alpha = \operatorname{span}(x_1,...,x_r)$ and consider the series $$0\subsetneq \operatorname{span}(x_1) \subsetneq \operatorname{span}(x_1,x_2) \subsetneq ... \subsetneq \operatorname{span}(x_1,...x_r)=M_\alpha.$$ Note that each $x_i$ can be annihilated by the maximal ideal $\frakm$ of $S$ since multiplying $x_i$ with elements in $\frakm$ will increase the degree. Since $Sx_i$ is cyclic, we have $Sx_i \cong S/\frakm$. Therefore each quotient of the above series is isomorphic to $S/\frakm$, so the series above is a composition series. Repeat this procedure for each graded piece of $M$ we get a composition series of $M$ and that $\ell(M) = \Dim_\mathbb{C}(M)$. 

	On the other hand if $M$ has infinitely many graded pieces over $\mathbb{C}$ so that $\Dim_\mathbb{C}(M) = \infty$, then the above argument shows that we can form a composition series of infinite length, and so $\ell(M)=\infty$.
\end{proof}

\begin{prop}\label{length_ext_fin}
	In our setting $\ell(\Ext^j_S(S/I^d,S)) < \infty$ and is nonzero if and only if $j=n(m-n)+1$ which corresponds to $s=0$ in Theorem \ref{schur-decompos}, and $d\geq n$. 
\end{prop}

\begin{proof}
	The correspondence of the cohomological index and $s$ can be seem in the proof of Lemma \ref{vanishing_degree}, and the condition $d\geq n$ can be seen from Lemma \ref{vanishing_degree} as well. Observe that the decomposition (\ref{decompose_first}) is finite, so we need to consider the decomposition of each $\Ext^j_S(J_{(\underline{z},l)},S)$. Suppose $s=0$. Then we have the restriction 
	\begin{align*}
		\begin{cases}
			\lambda_n=n-1-{z_{n-1}}-m,\\
			\lambda_1 \leq -m.
		\end{cases}
	\end{align*}
	Therefore in this case the set $W(\underline{z},n-1,(n-1),0)$ is bounded above by $(-m,...,-m,n-1-z_{n-1}-m)$ and below by $(-m,n-1-z_{n-1}-m,...,,n-1-z_{n-1}-m)$ and so is a finite set. Thus $\Ext^j_S(J_{(\underline{z},l)},S)$ can be decomposed as a finite direct sum of $S_{\lambda(s)}\mathbb{C}^m \otimes S_\lambda\mathbb{C}^n$ for $\lambda \in W(\underline{z},n-1,(n-1),0)$. By Proposition \ref{dim_schur} it is clear that the dimension of each $S_{\lambda(s)}\mathbb{C}^m \otimes S_\lambda\mathbb{C}^n$ is finite. So by Lemma \ref{length=dim}, $\ell(\Ext^j_S(J_{(\underline{z},l)},S)) = \Dim_\mathbb{C}(\Ext^j_S(J_{(\underline{z},l)},S)) < \infty$.

	On the other hand suppose $s\neq 0$. Then we have the restriction 
	\begin{align*}
		\begin{cases}
			\lambda_n=n-1-{z_{n-1}}-m,\\
			\lambda_s \geq s-n, \lambda_{s+1} \leq s-m.
		\end{cases}
	\end{align*}
	Since $\lambda_s \geq s-n$ implies that any weight that is greater than $(s-n,...,s-n,s-m,,...,s-m,n-1-z_{n-1}-m)$ is in $W(\underline{z},n-1,(n-1),s)$, the set $W(\underline{z},n-1,(n-1),s)$ is infinite, and therefore the decomposition of $\Ext^j_S(J_{(\underline{z},l)},S)$ in this case is infinite. So by Lemma \ref{length=dim} again we have $\ell(\Ext^j_S(J_{(\underline{z},l)},S)) = \Dim_\mathbb{C}(\Ext^j_S(J_{(\underline{z},l)},S)) = \infty$. 
	Therefore $\ell(\Ext^j_S(S/I^d,S)) < \infty$ if and only if $j=n(m-n)+1$.
\end{proof}

\begin{corollary}\label{sum_of_length}
	Let $j=n(m-n)+1$. Then we have $$\ell(\Ext^j_S(S/I^D,S)) = \sum^{D}_{d=n}\ell(\Ext^j_S(I^{d-1}/I^d,S)).$$
\end{corollary}
\begin{proof}
	Given the short exact sequence $$0\rightarrow I^{d-1}/I^d \rightarrow S/I^d \rightarrow S/I^{d-1} \rightarrow 0$$ we have the induced long exact sequence of $\Ext$-modules 
	\begin{align*}
		... \rightarrow &\Ext^{j-1}_S(I^{d-1}/I^d,S) \rightarrow \Ext^j_S(S/I^{d-1},S) \rightarrow \Ext^j_S(S/I^d,S)\\ \rightarrow &\Ext^j_S(I^{d-1}/I^d,S) \rightarrow \Ext^{j+1}_S(S/I^{d-1},S) \rightarrow ...
	\end{align*}
	By  Proposition \ref{isom_ext_modules} and Lemma \ref{injectivity} the map $\Ext^j(S/I^{d-1},S) \rightarrow \Ext^j(S/I^d,S)$ from the above long exact sequence is injective. Therefore we can split the above long exact sequence into short exact sequences $$0\rightarrow \Ext^j_S(S/I^{d-1},S) \rightarrow \Ext^j_S(S/I^d,S) \rightarrow \Ext^j_S(I^{d-1}/I^d,S) \rightarrow 0.$$ 
	By Lemma 3.3, $\Ext^j_S(S/I^d,S) = 0$ for $d<n$, so $\Ext^j_S(I^{d-1}/I^d) = 0$ for $d<n$ as well. Then by Proposition \ref{length_ext_fin} we have 
	\begin{align*}
		\begin{gathered}
			\ell(\Ext^j_S(I^{d-1}/I^d,S)) = \ell(\Ext^j_S(S/I^d,S)) - \ell(\Ext^j_S(S/I^{d-1},S)) \implies  \\
			\sum^D_{d=2} \ell(\Ext^j_S(I^{d-1}/I^d,S) = \ell(\Ext^j_S(S/I^D,S))-\ell(\Ext^j_S(S/I,S) \xRightarrow{\text{Lemma 3.3}}\\
			\sum^D_{d=n}\ell(\Ext^j_S(I^{d-1}/I^d,S)) = \ell(\Ext^j_S(S/I^D,S)),
		\end{gathered}
	\end{align*}
	as desired.
\end{proof}

\section{Multiplicites of the maximal minors}
In this section we will prove the main result for maximal minors. We recall the statement here, and recall that $I := I_n$.

\begin{theorem}\label{formula_limit}
Under the setting as in section 3, we have
\medspace
	\begin{enumerate}
		\item $j \neq n^2-1$ then $\ell(H^j(I^{d-1}/I^d))$ and $\ell(H^j(S/I^D))$ are either zero or infinite.
		\item If $j=n^2-1$ then $\ell(H^j_\frakm(I^{d-1}/I^d))$ and $\ell(H^j_\frakm(S/I^D))$ are finite and nonzero for $d$ and $D$ greater than $n$.  Moreover we have $$\frakj^j(I) = (mn)!\prod^{n-1}_{i=0}\dfrac{i!}{(m+i)!}$$ 
		\item In fact $$\frakj^j(I) = \epsilon^j(I).$$
	\end{enumerate}
\end{theorem}

\begin{remark}\label{Grassmannian}
	As mentioned in the introduction, this formula has a geometric interpretation. Recall that the degree of the Grassmannian $G(a,b)$ is
	\begin{align*}
		\operatorname{deg}(G(a,b)) = (a(b-a))! \prod^{a-1}_{i=0}\dfrac{i!}{(b-a+i)!},
	\end{align*}
	see \cite[Ch 4]{EisenbudHarris}. Replacing $a$ with $n$ and $b$ with $m+n$, we get 
	\begin{align*}
		\operatorname{deg}(G(n,m+n)) = (mn)! \prod^{n-1}_{i=0}\dfrac{i!}{(m+i)!},
	\end{align*}
	which is precisely $\epsilon^{n^2-1}(I_n)$ (and $\frakj^{n^2-1}(I_n)$), and so it must be an integer. 
\end{remark}

We will first prove the existence of $\frakj^j(I)$, then use it to prove the existence of $\epsilon^j(I)$. After that we will dicuss their formulae.

\begin{prop}\label{existence_limit}
	Let $$C=(mn-1)!\prod_{1\leq i \leq n}\dfrac{1}{(n-i)!(m-i)!}$$ and let $\delta=\{0\leq x_{n-1} \leq ... \leq x_1 \leq 1\}\subseteq \mathbb{R}^{n-1}$, $dx=dx_{n-1}...dx_1$. Then $\ell(H^{n^2-1}_\frakm(I^{d-1}/I^d)) < \infty$ and is nonzero for $d\geq n$. Moreover the limit $$\frakj^{n^2-1}(I) = \lim_{d\rightarrow \infty}\dfrac{(mn-1)!\ell(H^{n^2-1}_\frakm(I^{d-1}/I^d))}{d^{mn-1}}$$ exists and the formula is given by
	\begin{align}\label{integral_formula}
			 C\bigintss_\delta(\prod_{1\leq i \leq n-1}(1-x_i)^{m-n}x_i^2)(\prod_{1\leq i < j\leq n-1}(x_i-x_j)^2)dx.
\end{align}
\end{prop}

Before we give the proof of the above Proposition, we need to state some well-known results. We will use the local duality to study $\ell(H^j_\frakm(I^{d-1}/I^d))$ and $\ell(H^j_\frakm(S/I^D))$. Let $M^\vee$ denote the graded Matilis dual of an $R$-module $M$ where $R$ is a polynomial ring over $\mathbb{C}$ such that $$(M^\vee)_\alpha := \Hom_\mathbb{C}(M_-\alpha,\mathbb{C}),$$
and recall that the Matlis duality preserves length of finite length module.

\begin{lemma}\label{isom_of_length}
	Let $M$ be a finite length module over $S$. Then we have $$\ell(\Ext^j_S(M,S)) = \ell(H^{\Dim(S)-j}_\frakm(M)).$$
\end{lemma}
\begin{proof}
	By the local duality (see \cite[theorem 3.6.19]{BrunsHerzog}), we have $$\Ext^j_S(M,S(-\Dim(S))) \cong H^{\Dim(S)-j}_\frakm(M)^\vee.$$ Then the assertion of our lemma is immediate.
\end{proof}
Using this lemma we turn the problem into studying the length of $\Ext$-modules of cohomological degree $n(m-n)+1$. 
In the proof of Theorem \ref{existence_limit}, We will employ part of the strategy used in \cite{Kenkel}. However we will not resort to binomial coefficients since they will be too complicated to study in higher dimensional rings. We will instead use the following elementary but powerful facts.

\begin{theorem}(Euler-Maclaurin formula, see \cite{Apostol})\label{EM}
	Suppose $f$ is a function with continuous derivative on the interval $[1,b]$, then 
	$$\sum^b_{i=a}f(i) = \int^b_a f(x)dx+\dfrac{f(b)+f(a)}{2}+\sum^{\lfloor p/2 \rfloor}_{k=1}\dfrac{B_{2k}}{(2k)!}(f^{(2k-1)}(b)-f^{(2k-1)}(a))+R_p$$ where $B_{2k}$ is the Bernoulli number and $R_p$ is the remaining term.
\end{theorem}
For our application we only need to use the integral part on the RHS of the above formula. A well-known consequence is the following.
\begin{corollary}(Faulhaber's formula)\label{Faulhaber}
	The closed formula of the sum of $p$-th power of the first $b$ integers can be written as 
	\begin{align*}
		\sum^b_{k=1}k^p = \dfrac{b^{p+1}}{p+1}+\dfrac{1}{2}b^p+\sum^p_{k=2}\dfrac{B_k}{k!}\dfrac{p!}{p-k+1!}b^{p-k+1}.
	\end{align*}
Again the $B_k$ is the Bernoulli number. In particular, the sum on the LHS can be expressed as a polynomial of degree $p+1$ in $b$ with leading coefficient $\dfrac{1}{p+1}$.
\end{corollary}
\begin{proof}[Proof of Proposition \ref{existence_limit}]
	Let $s$ be as in Theorem \ref{schur-decompos}. By Lemma \ref{vanishing_degree}, we have $\Ext^j_S(I^{d-1}/I^d,S)\neq 0$ for $s=0$, so $\frakj^j(I) \neq 0$. The first claim follows from Proposition \ref{length_ext_fin} and Lemma \ref{isom_of_length}. We will prove the second claim. We first consider the length of $\Ext^j_S(I^{d-1}/I^d,S)$. By Lemma \ref{injectivity}, in order to calculate $\ell(\Ext^j_S(I^{d-1}/I^d,S)$ we only need to calculate the dimension of the tensor products of Schur modules that is in $\Ext^j_S(S/I^d)$ but not in $\Ext^j_S(S/I^{d-1},S)$. By Lemma \ref{weights_max_minors}, we need to consider the $\underline{z} \in \mathcal{P}_n$ such that $\{z_1=...=z_n=d-1\}$. This means we are considering the weights 
\begin{align*}
	\begin{cases}
		\lambda_n=n-d-m,\\
		\lambda_1\leq -m.
	\end{cases}
\end{align*}
i.e. $$\Ext^j_S(I^{d-1}/I^d,S) = \bigoplus S_{\lambda(0)}\mathbb{C}^m \otimes S_\lambda \mathbb{C}^n$$ where $\lambda$ satisfies the above conditions.
Adopting the notations of \cite{Kenkel}, we can write 
\begin{align*}
	\lambda &= (\lambda_1,\lambda_2,...\lambda_n)\\
		&= (\lambda_n+\epsilon_1,\lambda_n+\epsilon_2,...,\lambda_n)\\
		&= (n-d-m+\epsilon_1, n-d-m+\epsilon_2,...,n-d-m).
\end{align*}
Since $\lambda_1\leq -m$, it follows that $n-d-m \leq n-d-m+\epsilon_1 \leq m \implies 0\leq \epsilon_1 \leq d-n$. Since $\lambda$ is dominant, we have $0\leq \epsilon_{n-1}\leq ... \leq \epsilon_1 \leq d-n$. By Corollary \ref{same_dim}, we have $$\Dim(S_\lambda\mathbb{C}^n) = \Dim(S_{(\epsilon_1,...,\epsilon_{n-1},0)}\mathbb{C}^n$$ by adding $((n-d-m)^n)$ to $\lambda$. Therefore the dimension of $S_\lambda \mathbb{C}^n$ is given by 
\begin{align}\label{dim_small}
	\Dim(S_\lambda \mathbb{C}^n) = \Dim(S_{(\epsilon_1,...,\epsilon_{n-1},0)}\mathbb{C}^n) = (\prod_{1\leq i < j\leq n-1}\dfrac{\epsilon_i - \epsilon_j + j-i}{j-i})(\prod_{1\leq i \leq n-1}\dfrac{\epsilon_i+n-i}{n-i}).
\end{align}

Now we look at $S_{\lambda(0)}\mathbb{C}^m$. By definition $\lambda(0)=((-m)^{m-n},\lambda_1,...,\lambda_n)$. Use Corollary \ref{same_dim} again by adding $((n-d-m)^m)$ to $\lambda(0)$ we get that
\begin{align}\label{dim_large}
	\begin{split}
		\Dim(S_{\lambda(0)}\mathbb{C}^m) &= \Dim(S_{((d-n)^{(m-n)},\epsilon_1,...,\epsilon_{n-1},0)}\mathbb{C}^m\\
						 &= (\prod_{1\leq i\leq n-1}\dfrac{j-i}{j-i})(\prod_{1\leq i \leq m-n}\dfrac{d-\epsilon_1+m-2n+1-i}{m-n+1-i})\\
						 & (\prod_{1\leq i\leq m-n}\dfrac{d-\epsilon_2+m-2n+2-i}{m-n+2-i})(\epsilon_1-\epsilon_2+1)\\
						 &...\\
						 & (\prod_{1\leq i\leq m-n}\dfrac{d-n+m-i}{m-i})(\epsilon_{n-1}+1)(\dfrac{\epsilon_{n-2}+2}{2})...(\dfrac{\epsilon_1+n-1}{n-1}).
	\end{split}
\end{align}
Multiplying (\ref{dim_small}) and (\ref{dim_large}) we get that
\begin{align}\label{dim_multiply}
	\begin{split}
		\Dim(S_{\lambda(0)}\mathbb{C}^m \otimes S_\lambda\mathbb{C}^n)
		&= \Dim(S_{\lambda(0)}\mathbb{C}^m) \times \Dim(S_\lambda\mathbb{C}^n) \\
		&= \big(\prod_{1\leq i \leq m-n}\dfrac{d-n+m-i}{m-i}\big)\\
		&\Big(\big(\prod_{1\leq i \leq m-n}\dfrac{d-\epsilon_1-2n+m+1-i}{m-n+1-i}\big)\big(\dfrac{\epsilon_1+n-1}{n-1}\big)^2\\
		&\big(\prod_{1\leq i \leq m-n}\dfrac{d-\epsilon_2-2n+m+2-i}{m-n+2-i}\big)\big(\epsilon_1-\epsilon_2+1)^2\big(\dfrac{\epsilon_2+n-2}{n-2}\big)^2\\
		&...\\
		&\big(\prod_{1\leq i\leq m-n}\dfrac{d-n-\epsilon_{n-1}+m-1-i}{m-1-i}\big)(\epsilon_{n-2}-\epsilon_{n-1}+1)^2...(\epsilon_{n-1}+1)^2\Big)
	\end{split}
\end{align}
The formula (\ref{dim_multiply}) is for a particular choice of $\epsilon_1,...,\epsilon_{n-1}$. To calculate $\ell(\Ext^j_S(I^{d-1}/I^d,S))$ we need to add the result of all possible choices of $\epsilon_1,...,\epsilon_{n-1}$. After some manipulations we will end up with 
\begin{align}\label{dim_formula_sep}
	\ell(\Ext^j_S(I^{d-1}/I^d,S) &= \sum_{0\leq \epsilon_{n-1}\leq ... \leq \epsilon_1\leq d-n}(\ref{dim_multiply})\\
	\tag{\ref{dim_formula_sep} - 0} &= (\prod_{1\leq i\leq m-n}\dfrac{d-n+m-i}{m-i}) \\ 
	\tag{\ref{dim_formula_sep} - 1}	     &\Big(\sum^{d-n}_{\epsilon_1=0}\big(\prod_{1\leq i \leq m-n}\dfrac{d-\epsilon_1-2n+m+1-i}{m-n+1-i}\big)\big(\dfrac{\epsilon_1+n-1}{n-1}\big)^2\\
	\tag{\ref{dim_formula_sep} - 2}			     &(\sum^{\epsilon_1}_{\epsilon_2=0}\big(\prod_{1\leq i\leq m-n}\dfrac{d-\epsilon_2-2n+m+2-i}{m-n+2-i}\big)(\epsilon_1-\epsilon_2+1)^2\big(\dfrac{\epsilon_2+n-2}{n-2}\big)^2\\
	\notag					     &...\\
\tag{\ref{dim_formula_sep} - (n-2)}	     &(\sum^{\epsilon_{n-3}}_{\epsilon_{n-2}=0}\big(\prod_{1\leq i\leq m-n}\dfrac{d-n-\epsilon_{n-2}+m-2-i}{m-2-i}\big)(\epsilon_{n-3}-\epsilon_{n-2}+1)^2...(\dfrac{\epsilon_{n-2}+2}{2})^2\big)\\
	\tag{\ref{dim_formula_sep} - (n-1)}	     &(\sum^{\epsilon_{n-2}}_{\epsilon_{n-1}=0}\big(\prod_{1\leq i\leq m-n}\dfrac{d-n-\epsilon_{n-1}+m-1-i}{m-1-i}\big)(\epsilon_{n-2}-\epsilon_{n-1}+1)^2...(\epsilon_{n-1}+1)^2\big)...\Big)
\end{align}
Now Corollary \ref{Faulhaber} shows that the above sum will be a polynomial in $d$, and we need to calculate its degree. Corollary \ref{Faulhaber} also implies that when looking at each sum of (\ref{dim_formula_sep}) we only need to look at the summands that will contribute to the highest degree of the resulting polynomial. We see that the sum (\ref{dim_formula_sep} - (n-1)) can be expressed as a degree $m-n+2(n-1)+1 = m+n-1$ polynomial in $\epsilon_{n-2}$. Similarly (\ref{dim_formula_sep} - (n-2)) can be expressed as a degree $2m+2n-4$ polynomial in $\epsilon_{n-3}$. Continuing in this fashion we see that the sum (\ref{dim_formula_sep} - 1) can be expressed as a degree $mn-m+n-1$ polynomial in $d$. Multiplying (\ref{dim_formula_sep} - 0) with (\ref{dim_formula_sep} - 1) will result in a degree $mn-1$ polynomial.

Moreover, after factoring out the coefficients of the terms that will eventually contribute to the highest degree of the resulting polynomial of (\ref{dim_formula_sep}) and then apply Theorem \ref{EM} to the sum of said terms, the leading coefficient of the resulting polynomial of (\ref{dim_formula_sep}) is given by  
\begin{align*}
		\begin{gathered}
			\prod_{1\leq i\leq n}\dfrac{1}{(n-i)!(m-i)!} \lim_{d\rightarrow \infty}\dfrac{\bigintss_\Delta(\prod_{1\leq i \leq n-1}(d-x_i)^{m-n})(\prod_{1\leq i \leq n-1}x_i^2)(\prod_{1\leq i < j\leq n-1}(x_i-x_j)^2)dA}{d^{mn-m+n-1}},\\
			\Delta = \{0 \leq x_{n-1} \leq ... \leq x_1 \leq d-n \}.
\end{gathered}
\end{align*}
where the factor $\frac{1}{(m-1)!(n-1)!}$ comes from $(\ref{dim_formula_sep} - 0)$ and the coefficients of $(\dfrac{\epsilon_i}{n-i})^2$, and the product $\prod_{2\leq i\leq n}\frac{1}{(n-i)!(m-i)!}$ comes from the coefficients of the needed terms from the rest of $(\ref{dim_formula_sep})$. Since the above limit exists and the integrand is a homogeneous polynomial in $d,x_1,...,x_{n-1}$, we can simplify it to
\begin{align}\label{formula_mid}
		\begin{gathered}
			\prod_{1\leq i\leq n}\dfrac{1}{(n-i)!(m-i)!} \bigintss_\delta(\prod_{1\leq i \leq n-1}(1-x_i)^{m-n}x_i^2)(\prod_{1\leq i < j\leq n-1}(x_i-x_j)^2)dA,\\
			\delta = \{0 \leq x_{n-1} \leq ... \leq x_1 \leq 1 \}.
\end{gathered}
\end{align}
Multiplying the result with $(mn-1)!$ completes the proof.
\end{proof}

We will give some examples of the above formula.

\begin{example}
	Let $n=2$ and $j=3$. By Lemma \ref{vanishing_degree} and Lemma \ref{length_ext_fin}, $H^3_\frakm(I^{d-1}/I^d)\neq 0$ and has finite length. The integral we need to calculate is simply $$\int^{1}_{0}(1-x_1)^{m-2}x_1^2dx_1 = \dfrac{2}{m^3-m}.$$ Since $C=\frac{(2m-1)!}{(m-1)!(m-2)!}$, we get that $$\frakj^3(I) = \dfrac{1}{(m+1)!m!}(2m)! = \dfrac{1}{m+1}\binom{2m}{m}.$$ This recovered the result from \cite[Corollary 1.2]{Kenkel}.
\end{example}

\begin{example}
	Let $n=3$ and $j=8$. Again one can check with Lemma \ref{vanishing_degree} and \ref{length_ext_fin} that $H^8_\frakm(I^{d-1}/I^d)$ and $H^8_\frakm(S/I^D)$ are nonzero and has finite length for $D > d \geq n$. By Proposition \ref{existence_limit} we first calculate the integral $$\int_{0\leq x_2 \leq x_1 \leq 1}(1-x_1)^{m-3}(1-x_2)^{m-3}x_1^2x_2^2(x_1-x_2)^2 dx_2dx_1.$$ This can be done by doing integration by parts multiple times or simply use Sage. The result is $\frac{12}{m^2(m^2-4)(m^2-1)^2}$. 

	We also have $C=(3m-1)!\frac{1}{(m-3)!}\frac{1}{(m-2)!}\frac{1}{2(m-1)!}$. Therefore 
\begin{align*}
	\frakj^8(I) &= (3m-1)!\dfrac{12}{m^2(m^2-4)(m^2-1)^2(m-3)!(m-2)!2(m-1)!}\\
		      &= (3m)!\dfrac{2}{(m+2)!(m+1)!m!}.
\end{align*}

More specifically, consider the case when $m=4$. Then by Lemma \ref{vanishing_degree} and Proposition \ref{length_ext_fin}, $m-n=1$ and $n(m-n)+1 = 4$. We get
\begin{enumerate}
	\item The non-vanishing cohomological degrees of $\Ext^j_S(I^{d-1}/I^d,S)$ are $j=2,3,4$.
	\item Only $\Ext^4_S(I^{d-1}/I^d,S)$ is nonzero and has finite length.
	\item $\frakj^8(I) = (12)!2/(4!5!6!) = 462$.
\end{enumerate}

When $m=5$, $m-n = 2$ and $n(m-n)+1 = 7$.
\begin{enumerate}
	\item The non-vanishing cohomological degrees of $\Ext^j_S(I^{d-1}/I^d,S)$ are $j=3, 5, 7$.
	\item Only $\Ext^7_S(I^{d-1}/I^d,S)$ is nonzero and has finite length.
	\item $\frakj^8(I) = (15)!2/(5!6!7!) = 6006$.
\end{enumerate}
\end{example}

The examples above hinted that $\frakj^{n^2-1}(I)$ should be $(mn)!\prod^{n-1}_{i=0}\dfrac{i!}{(m+i)!}$ as stated in the main theorem, and we will prove that this is indeed the case. We first recall a classical result of Atle Selberg. For English reference one might check \cite[(1.1)]{ForresterWarnaar}.
\begin{theorem}(See \cite{Selberg})\label{Selberg_int_thm} 
	For $a, b$ and $c$ in $\mathbb{C}$ such that $\operatorname{Re}(a) > 0$, $\operatorname{Re}(b) > 0$ and $\operatorname{Re}(c) > -\operatorname{min}\{1/n, \operatorname{Re}(a)/(n-1), \operatorname{Re}(b)/(n-1)\}$ we have 
	\begin{align*}\label{Selberg_int}
		S_n(a,b,c) &= \bigintss_{[0,1]^n}\prod^n_{i=1}x_i^{a-1}(1-x_i)^{b-1}\prod_{1\leq i < j \leq n}|x_i-x_j|^{2c}dA \\
			   &= \prod^{n-1}_{i=0} \dfrac{\Gamma(a+ic)\Gamma(b+ic)\Gamma(1+(i+1)c)}{\Gamma(a+b+(n+i-1)c)\Gamma(1+c)}.
	\end{align*}
	where $\Gamma$ is the usual Gamma function $\Gamma(k)=(k-1)!$.
\end{theorem}
Now we can prove Theorem \ref{formula_limit}.

\begin{proof}[Proof of Theorem \ref{formula_limit}]	
	(1) Follows from Lemma \ref{vanishing_degree}, Proposition \ref{length_ext_fin} and Lemma \ref{isom_of_length}.

	(2) By Proposition \ref{existence_limit} it remains to evaluate $$C\bigintss_\delta \prod^{n-1}_{i=1}(1-x_i)^{m-n}x_i^2\prod_{1\leq i < j \leq n-1}(x_i-x_j)^2 dA$$ where $$C = (mn-1)!\prod^n_{i=1}\dfrac{1}{(m-i)!(n-i)!},  \delta = \{0 \leq x_{n-1} \leq ... \leq x_1 \leq 1\}.$$

	By Theorem \ref{Selberg_int_thm} we have 
	\begin{align*}
		& C\bigintss_{[0,1]^{n-1}}\prod^{n-1}_{i=1}(1-x_i)^{m-n}x_i^2\prod_{1\leq i < j \leq n-1}(x_i-x_j)^2 dA
	     \\ &= C \prod^{n-2}_{i=0}\dfrac{\Gamma(3+i)\Gamma(m-n+1+i)\Gamma(2+i)}{\Gamma(m+i+2)\Gamma(2)} \\ &= C \prod^{n-2}_{i=0}\dfrac{(2+i)!(m-n+i)!(1+i)!}{(m+i+1)!}
	     \\ &= \dfrac{(mn)!}{mn}\dfrac{1}{(m-n)!}\prod^{n-1}_{i=1}\dfrac{1}{(m-i)!(n-i)!}\prod^{n-1}_{i=1}\dfrac{(1+i)!(m-n+i-1)!(i)!}{(m+i)!}
	     \\ &= \dfrac{(mn)!}{mn}\dfrac{1}{(m-n)!}\prod^{n-1}_{i=1}\dfrac{(1+i)!}{(m-n+i)!(m-n+i)...(m+i)}
	     \\ &= \dfrac{(mn)!}{mn}\dfrac{1}{(m-n)!}\dfrac{(m-n)!}{(m-1)!}\prod^{n-1}_{i=1}\dfrac{(1+i)!}{(m+i)!}
	     \\ &= \dfrac{(mn)!}{n}\prod^{n-1}_{i=0}\dfrac{1}{(m+i)!}\prod^{n-1}_{i=1}(1+i)!
	\end{align*}
	Since the integrand $\prod^{n-1}_{i=1}(1-x_i)^{m-n}x_i^2\prod_{1\leq i < j \leq n-1} (x_i-x_j)^2$ does not change under permutation of variables, we have 
	\begin{align*}
		& \bigintss_{[0,1]^{n-1}}\prod^{n-1}_{i=1}(1-x_i)^{m-n}x_i^2\prod_{1\leq i < j \leq n-1}(x_i-x_j)^2 dA 
		\\ = & (n-1)!\bigintss_\delta \prod^{n-1}_{i=1}(1-x_i)^{m-n}x_i^2\prod_{1\leq i < j \leq n-1}(x_i-x_j)^2 dA
	\end{align*}
	Hence we have 
	\begin{align*}
		 \frakj^{n^2-1}(I) & = C\bigintss_\delta \prod^{n-1}_{i=1}(1-x_i)^{m-n}x_i^2\prod_{1\leq i < j \leq n-1}(x_i-x_j)^2 dA \\
		& = \dfrac{1}{(n-1)!}\dfrac{(mn)!}{n}\prod^{n-1}_{i=0}\dfrac{1}{(m+i)!}\prod^{n-1}_{i=1}(1+i)! \\
		& = (mn)!\prod^{n-1}_{i=0}\dfrac{i!}{(m+i)!}.
	\end{align*}

	(3) Let $j=mn-n^2+1$. By Corollary \ref{sum_of_length} we need to sum $\ell(\Ext^j_S(I^{d-1}/I^d,S))$ over all $n \leq d\leq D$ to get $\ell(\Ext^j_S(S/I^D,S))$. It is clear that by Corollary \ref{Faulhaber} the sum 
\begin{equation}\label{sum_all_d}
\ell(\Ext^j_S(S/I^D,S))=\sum^D_{d=n}\ell(\Ext^j_S(I^{d-1}/I^d,S))
\end{equation}
can be expressed as a polynomial in $D$ of degree $mn$. By Lemma \ref{isom_of_length} we see that $\ell(H^{mn-j}_\frakm(S/I^D))$ is a polynomial in $D$ of degree $mn$ as well. Therefore we have $$\epsilon^{mn-j}(I) = \lim_{D\rightarrow \infty}\dfrac{(mn)!\ell(H^{mn-j}_\frakm(S/I^D))}{D^{mn}} < \infty,$$ where $mn-j = mn-n(m-n)-1 = n^2-1$.

Finally, apply Corollary 4.4 to (\ref{sum_all_d}), we see that the leading coefficient of the resulting polynomial of (\ref{sum_all_d}) is given by multiplying $1/mn$ to $(\ref{formula_mid})$, then multiplying the result with $(mn)!$ yields the desired formula, which is precisely $\epsilon^{n^2-1}(I) = \frakj^{n^2-1}(I)$. 
\end{proof}

%

\section{Decompositions of Ext modules of sub-maximal Pfaffians}
We will follow the same strategies to prove the existence of the $j$-multiplicities of the $\Ext$-module of the Pfaffians for a suitable $j$. Let $\Pf_{2k}$ be the $2k \times 2k$ Pfaffian of $S = \Sym(\bigwedge^2 \mathbb{C}^n)$ which can be considered as a polynomial ring with variables in a skew-symmetric matrix. In this section we recall the result of the decomposition of $\Ext^\bullet_S(S/\Pf_{2k}^d,S)$ from \cite{Perlman}.

We first recall some notations from \cite{Perlman}. Recall that $\mathcal{P}(k) = \{\underline{z} = (z_1 \geq ... \geq z_k \geq 0)\}$ and $\mathcal{P}_e(k)$ the partitions with columns of even lengths. We denote $$\underline{z}^{(2)} = (z_1,z_1,z_2,z_2,...,z_k,z_k) \in \mathcal{P}_e(2k).$$ 

It is well-known that $$S=\Sym(\bigwedge^2 \mathbb{C}^n) = \bigoplus_{\underline{z}\in \mathcal{P}(m)}S_{\underline{z}^{(2)}}\mathbb{C}^n,$$ see for example \cite[Proposition 2.3.8]{Weyman}. In \cite{AbeasisDelFra} the authors classified the $\GL$-invariant ideals in $S$. As in the case of generic matrix, we can consider the ideal $I_{\underline{x}}$ generated by $S_{\underline{x}^{(2)}}\mathbb{C}^n$. It can be shown that $I_{\underline{z}} = \bigoplus_{\underline{y} \geq \underline{z}} S_{\underline{y}^{(2)}}\mathbb{C}^n$. Again the $\GL$-invariant ideals in $S$ can be written as $$I_\mathcal{X} = \bigoplus_{\underline{x} \in \mathcal{X}} I_{\underline{x}}$$ for $\mathcal{X} \subset \mathcal{P}(m)$. Recall that we denote $$\mathcal{X}^d_p = \{\underline{x} \in \mathcal{P}(m) : |\underline{x}| = kd, x_1 \leq d\},$$ and it was shown in \cite{AbeasisDelFra} that $\Pf_{2p}^d = I_{\mathcal{X}^d_p}$. Now we are in position to state the main tool for pfaffians. Note that the result in \cite{Perlman} was stated in terms of the dual vector space $(\mathbb{C}^n)^*$, and we will follow this convention here.

\begin{theorem}(see \cite[Theorem F, Theorem 3.3]{Perlman})\label{perlman_main}
Let 
\begin{equation}\label{Tl(z)}
	\mathcal{T}_l(\underline{z}) = \{\underline{t} = (l = t_1 \geq ... \geq t_{n-2l}) \in \mathbb{Z}^{n-2l}_{\geq 0} | \substack{z^{(2)}_{2l+i} - z^{(2)}_{2l+i+1}\geq 2t_i - 2t_{i+1}\\ 1\leq i \leq n-2l-1}\}
\end{equation}

and let $W(\underline{z},l,\underline{t})$ denote the set of dominant weights $\lambda$ satisfying the following conditions:

\begin{align}\label{weight_restrict_pfaffian}
\begin{cases}
	\lambda_{2l+i-2t_i} = z_{2l+i}^{(2)} + n - 1 -2t_i, i=1,...,n-2l,\\
	\lambda_{2i} = \lambda_{2i-1}, 0 < 2i < n-2t_{n-2l},\\
	\lambda_{n-2i} = \lambda_{n-2i-1}, 0 \leq i \leq t_{n-2l}-1.
\end{cases}
\end{align}
Then for each $j\geq 0$ we have
\begin{align}\label{decompose_ext_schur_pfaffian}
	\Ext^j_S(J_{\underline{z},l},S) = \bigoplus_{\substack{\underline{t} \in \mathcal{T}_l(\underline{z}) \\ \binom{n}{2} - \binom{2l}{2} - 2\sum^{n-2l}_{i=1}t_i = j \\ \lambda \in W(\underline{z}, l, \underline{t})}} S_\lambda (\mathbb{C}^n)^*
\end{align}
where $J_{\underline{z},l}$ is defined the same way as in section 3, and $S_\lambda (\mathbb{C}^n)*$ appears in degree $-|\lambda|/2$. 
Moreover, we have
\begin{align}\label{decompose_ext_pfaffian}
	\Ext^j_S(S/\Pf^d_{2p},S) = \bigoplus_{(\underline{z},l) \in \mathcal{Z}^d_p} \Ext^j_S(J_{\underline{z},l},S).
\end{align}
where $\mathcal{Z}^d_p$ is defined the same way as in Lemma \ref{z(x)_set}.
\end{theorem}

From now on we focus on the sub-maximal pfaffians. Let $S = \Sym(\bigwedge^2 \mathbb{C}^{2n+1})$ and $\Pf := \Pf_{2n}$.

\begin{lemma}\label{non_vanishing_ext_pfaffian}
	In our setting $\Ext^j_S(S/\Pf^d,S) \neq 0$ if and only if $j=2(n-t_3)+1$. Moreover, $\Ext^{2n+1}_S(S/\Pf^d,S) \neq 0$ if and only if $d\geq 2n-1$.
\end{lemma}

\begin{proof}
Recall that when $p=n$, by Lemma \ref{weights_max_minors} we have that $(\underline{z},l) = ((c^n),n-1) \in \mathcal{Z}^d_n$, and so for such $\underline{z}$ we have $\underline{z}^{(2)} = (c^{2n})$ for $0 \leq c \leq d-1$. Applying this information to (\ref{Tl(z)}) we see that $t_1 = t_2 = n-1$ and since $z^{(2)}_{2n+1} = 0$, we have $z^{(2)}_{2n} - z^{(2)}_{2n+1} \geq 2(t_2 - t_3) \implies t_3 \geq (2(n-1)-c) / 2$. Then we get that

\begin{align*}
	&\binom{2n+1}{2} - \binom{2n-2}{2} - 2(2n-2 + t_3) = j\\
	&\implies 2(n-t_3)+1 = j.
\end{align*}
Moreover, applying this information to (\ref{weight_restrict_pfaffian}), we get that $W(\underline{z}, l, t) = W((c^n),n-1,\underline{t})$ consist of weights of the form
\begin{align}\label{weights_restriction_pfaffian}
	\begin{cases}
		\lambda_1 = \lambda_2 = c+2 \leq d+1,\\
		\lambda_{2(n-t_3)+1} = 2n-2t_3,\\
		\lambda_{2i} = \lambda_{2i-1}, 0 \leq 2i < 2n+1-2t_3,\\
		\lambda_{2n+1-2i} = \lambda_{2n-2i}, 0 \leq i \leq t_3-1.
	\end{cases}
\end{align}

In particular, when $j=2n+1$, we have $t_3 = 0$ and
\begin{align}\label{weights_restriction_pfaffian_2}
	\begin{cases}
		\lambda_1 = \lambda_2 = c+2 \leq d+1,\\
		\lambda_{2i} = \lambda_{2i-1}, 0 \leq 2i < 2n+1,\\
		\lambda_{2n+1} = 2n.
	\end{cases}
\end{align}
Since the weights are dominant, $2n \leq d+1 \implies 2n-1 \leq d$. 
\end{proof}

\begin{prop}\label{finite_length_pfaffian}
	In our setting, $\ell(\Ext^j_S(S/\Pf^d,S)) < \infty$ and is nonzero if and only if $j=2n+1$.
\end{prop}

\begin{proof}
we would like to identify the set $W(\underline{z},l,t)$ that is finite. For each $d$ we have $\lambda_1 = \lambda_2 \leq d+1$, so we have an upper bound. The lower bound comes from $\lambda_{2n}$. We see from from (\ref{non_vanishing_ext_pfaffian}) that if $j\neq 2n+1$ then there is no lower bound for the set. On the other hand we see from (\ref{non_vanishing_ext_pfaffian}) that when $j=2n+1$ the set is finite and nonzero. This means by Lemma \ref{length=dim} the only $\Ext^j_S(S/I^d, S)$ with finite length is the one with $j=2n+1$.
\end{proof}

As in the case of maximal minors of generic matrix, we have the corresponding injectivity maps for the $\Ext$ modules.
\begin{lemma}(see \cite[Corollary 5.4]{RaicuWeymanWitt})\label{ext_pfaffian}
	In our setting, we have $$\Hom_S(\Pf^d,S) = S, \Ext^1_S(S/\Pf^d,S) = 0,$$ and $$\Ext^{j+1}_S(S/\Pf^d,S) = \Ext^j(\Pf^d,S)$$ for $j>0$.	
\end{lemma}

\begin{prop}\label{injection_ext_pfaffian}\cite[Theorem 5.5]{RaicuWeymanWitt}
	Given the short exact sequence $$0 \rightarrow \Pf^d \rightarrow \Pf^{d-1} \rightarrow \Pf^{d-1}/\Pf^d \rightarrow 0,$$ we have the induced injection map $$\Ext^j_S(\Pf^{d-1},S) \hookrightarrow \Ext^j_S(\Pf^d,S)$$ for any $j$ such that $\Ext^j_S(\Pf^d,S) \neq 0$.
\end{prop}

Combining the above results we get the following.

\begin{theorem}\label{sum_of_length_pfaffian}
	In our setting we have $$\ell(\Ext^j_S(S/\Pf^D,S)) = \sum^D_{d=2n-1}\Ext^j_S(\Pf^{d-1}/\Pf^d,S)$$ for $j=2n+1.$
\end{theorem}
\begin{proof}
	The argument is identical to Corollary \ref{sum_of_length} and follows from Lemma \ref{non_vanishing_ext_pfaffian}, Proposition \ref{finite_length_pfaffian} and Proposition \ref{injection_ext_pfaffian}.
\end{proof}

\section{Multiplicities of thickenings of sub-maximal Pfaffians}
In this section we will prove the main theorem for sub-maximal pfaffians.
\begin{theorem}\label{formula_limit_pfaffian}
	Under the same setting as in section 5, we have
	\begin{enumerate}
		\item If $j \neq 2n^2-n-1$ then $\ell(H^j_\frakm(\Pf^{d-1}/\Pf^d))$ and $\ell(H^j_\frakm(S/\Pf^D,S))$ are either zero or $\infty$.
		\item If $j = 2n^2-n-1$ then $\ell(H^j_\frakm(\Pf^{d-1}/\Pf^d))$ and $\ell(H^j_\frakm(S/\Pf^D,S))$ are finite and nonzero for $d$ and $D$ greater than $2n-1$. Moreover we have
			$$\epsilon^j(\Pf) = (2n^2+n)!\prod^{n-1}_{i=0}\dfrac{(2i)!}{(2n+1+2i)!}$$ 
		\item In fact
			$$\frakj^j(\Pf) = \epsilon^j(\Pf).$$
	\end{enumerate}
\end{theorem}

\begin{remark}\label{Orthogonal_Grass}
	As in the case of maximal minors (Remark \ref{Grassmannian}), there is a geometric interpretation of this formula. Let $OG(a,b)$ denote the Orthogonal Grassmannian. By the discussion of \cite[p83, p88]{Totaro}, we have
	\begin{align*}
		\operatorname{deg}(SO(2a+1)/U(a)) = \operatorname{deg}(OG(a,2a+1)) = ((a^2+a)/2)!\dfrac{1!2!...(a-1)!}{1!3!...(2a-1)!}.
	\end{align*}
	Replace $a$ by $2n$, we get 
	\begin{align*}
		\operatorname{deg}(OG(2n,4n+1)) = (2n^2+n)!\dfrac{1!2!...(2n-1)!}{1!3!...(4n-1)!}.
	\end{align*}
	After the cancellation, we end up with the same formula of $\epsilon^{2n^2-n-1}(\Pf)$, and thus $\epsilon^{2n^2-n-1}(\Pf)$ (and $\frakj^{2n^2-n-1}(\Pf)$) must be an integer. 
\end{remark}

Again we will first prove the existence of $\frakj^j(\Pf)$, then use it to prove the existence of $\epsilon^j(\Pf)$.

\begin{prop}\label{existence_limit_pfaffian_j}
	Let $$C = (2n^2+n-1)!\prod_{1\leq i \leq 2n} \dfrac{1}{i!},$$ and let $\delta = \{0 \leq x_{n-1} \leq ... \leq x_1 \leq 1\} \subset \mathbb{R}^{n-1}$, $dx = dx_{n-1...dx_1}$. Then $\ell(H^{2n^2-n-1}_\frakm(\Pf^{d-1}/Pf^d)) \leq \infty$ and is nonzero for $d \geq 2n-1$. Moreover the limit
		$$\frakj^{2n^2-n-1}(\Pf) = \lim_{d\rightarrow \infty}\dfrac{(2n^2+n-1)!\ell(H^{2n^2-n-1}(\Pf^{d-1}/\Pf^d))}{d^{2n^2+n-1}}$$
	exists and is equal to
	$$C\bigintsss_{\delta} \prod_{1\leq i < j\leq n-1}(x_i - x_j)^4\prod_{1\leq i \leq n-1}x_i^2(1-x_i)^4dx.$$
\end{prop}

\begin{proof}
Let $m:=\Dim(S) = 2n^2+n$ and fix $j=2n+1$. As in section 4, we first calculate $\ell(\Ext^j_S(\Pf^{d-1}/\Pf^d,S))$. Since by Lemma \ref{length=dim} this is equal to the dimension of $\Ext^j_S(\Pf^{d-1}/\Pf^d,S)$ as $\mathbb{C}$-vector space, we only need to consider the direct summand appears in (\ref{decompose_ext_schur_pfaffian}) with weight $\lambda$ satisfying $\lambda_1 = \lambda_2 = d+1$. Therefore we get $$\Ext^j_S(\Pf^{d-1}/\Pf^d,S) = \bigoplus S_\lambda (\mathbb{C}^{2n+1})^*$$ where by (\ref{weight_restrict_pfaffian}) $\lambda$ is of the form $\lambda = (d+1,d+1,\lambda_2,\lambda_2,...,\lambda_n,\lambda_n,2n)$, and we will calculate

\begin{align*}
	\ell(\Ext^j_S(\Pf^{d-1}/\Pf^d,S)) = \Dim(S_\lambda (\mathbb{C}^{2n+1})^*).
\end{align*}

Let $0\leq \epsilon_1 \leq ... \leq \epsilon_{n-1}\leq d+1-2n$, then we can rewrite $$\lambda = (d+1,d+1, 2n+\epsilon_1, 2n+\epsilon_1,...,2n+\epsilon_{n-1}, 2n+\epsilon_{n-1}, 2n).$$ 
Now by Corollary \ref{same_dim} again we have 
$$S_\lambda (\mathbb{C}^{2n+1})^* \cong S_{(d+1-2n,d+1-2n,\epsilon_1,\epsilon_1,...,\epsilon_{n-1},\epsilon_{n-1},0)}(\mathbb{C}^{2n+1})^*,$$
by Proposition \ref{dim_schur} we calculate, for a fixed set of $\epsilon_1,...,\epsilon_{n-1}$, the dimension of $S_\lambda(\mathbb{C}^{2n+1})^*$. For simplicity we let $c := d+1-2n$.

\begin{align}\label{dim_each_schur}
	\begin{split}
		&\Dim(S_\lambda (\mathbb{C}^{2n+1})^*) =\prod_{1\leq i < j\leq n-1}\dfrac{\lambda_i-\lambda_j + j-i}{j-i}=\\
	&\dfrac{c+2n+1-1}{2n+1-1}\dfrac{c+2n+1-2}{2n+1-2}\prod_{1\leq i\leq n-1}\dfrac{\epsilon_i+2n+1-(2i+1)}{2n+1-(2i+1)}\dfrac{\epsilon_i+2n+1-(2n+2)}{2n+1-(2i+2)}\\
	&\prod_{1\leq i \leq n-1}\big(\dfrac{c-\epsilon_i+(2i+1)-1}{(2i+1)-1}\dfrac{c-\epsilon_i+(2i+2)-1}{(2i+2)-1}\dfrac{c-\epsilon_i+(2i+1)-2}{(2i+1)-2}\dfrac{c-\epsilon_i+(2i+2)-2}{(2i+2)-2}\big)\\
	&\prod_{1\leq i<j \leq n-1}\dfrac{\epsilon_i-\epsilon_j+(2j+2)-(2i+2)}{(2j+2)-(2i+2)}\dfrac{\epsilon_i-\epsilon_j+(2j+1)-(2i+1)}{(2j+1)-(2i+1)}\\
	&\dfrac{\epsilon_i-\epsilon_j+(2j+2)-(2i+1)}{(2j+2)-(2i+1)}\dfrac{\epsilon_i-\epsilon_j+(2j+1)-(2i+2)}{(2j+1)-(2i+2)}
	\end{split}
\end{align}

Next, we need to add together (\ref{dim_each_schur}) for all possible $\epsilon_1,...,\epsilon_{n-1}$.

\begin{align}\label{sum_dim_schur_pfaffian}
	\begin{split}
	&\ell(\Ext^j_S(\Pf^{d-1}/\Pf^d,S)) \\
	&= \sum_{0\leq \epsilon_{n-1} \leq ... \leq \epsilon_1 \leq c} (\ref{dim_each_schur})\\
	&= \dfrac{c+2n+1-1}{2n+1-1}\dfrac{c+2n+1-2}{2n+1-2}\\
	&\sum_{\epsilon_1 = 0}^c \dfrac{\epsilon_1+2n-2}{2n-2}\dfrac{\epsilon_1+2n-3}{2n-3}\Big(\dfrac{c-\epsilon_{1}+2}{2}\dfrac{c-\epsilon_{1}+1}{1}\dfrac{c-\epsilon_1+3}{3}\dfrac{c-\epsilon_{1}+2}{2}\Big)\\
	&...\\
	&\sum_{\epsilon_{n-1}=0}^{\epsilon_{n-2}} \dfrac{\epsilon_{n-1}+1}{1}\dfrac{\epsilon_{n-1}+2}{2}\\
	&\Big(\dfrac{c-\epsilon_{n-1}+2n-2}{2n-2}\dfrac{c-\epsilon_{n-1}+2n-1}{2n-1}\dfrac{c-\epsilon_{n-1}+2n-3}{2n-3}\dfrac{c-\epsilon_{n-1}+2n-2}{2n-2}\Big)\\
	&\Big(\prod_{1\leq i \leq n-2}\dfrac{\epsilon_i-\epsilon_{n-1}+2n-(2i+2)}{2n-(2i+2)}\dfrac{\epsilon_i-\epsilon_{n-1}+2n-1-(2i+1)}{(2n-1)-(2i+1)}\\
	&\dfrac{\epsilon_i-\epsilon_{n-1}+2n-(2i+1)}{2n-(2i+1)}\dfrac{\epsilon_i-\epsilon_{n-1}+2n-1-(2i+2)}{(2n-1)-(2i+2)}\Big)
	\end{split}
\end{align}
An argument similar to one in the proof of Proposition \ref{existence_limit} shows that the above sum can be written as a polynomial of degree $2n^2+n-1$ in $d$.

Moreover, an argument similar to the one in the proof of Proposition \ref{existence_limit} shows that, using (\ref{sum_dim_schur_pfaffian}), $\frakj^{2n^2-n-1}(\Pf)$ can be written as
\begin{align}\label{limit_int_pfaffian}
	\begin{split}
	&(2n^2+n-1)!\dfrac{1}{(2n)(2n-1)}\dfrac{1}{(2n-2)!}\dfrac{1}{(2n-1)!(2n-2)!}\prod_{1\leq i \leq 2n-3}\dfrac{1}{i!}\cdot\\
	&\bigintsss_{\delta} \prod_{1\leq i < j\leq n-1}(x_i - x_j)^4\prod_{1\leq i \leq n-1}x_i^2(1-x_i)^4 dA
	\end{split}
\end{align}
where $\delta = \{0 \leq \epsilon_{n-1} \leq ... \leq \epsilon_1 \leq d+1-2n\}$, which, after simplification, is precisely the formula in our assertion. To be more precise, the factor $1/(2n-2)!$ comes from the terms of the form $\epsilon_i/(2i+1-(2n+1))$ and $\epsilon_i/(2i+2-(2n+1))$. The factor $1/(2n-1)!(2n-2)!$ comes from the terms $$\dfrac{c-\epsilon_i}{2i+1-1}\dfrac{c-\epsilon_i}{2i+1-2}\dfrac{c-\epsilon_i}{2i+2-1}\dfrac{c-\epsilon_i}{2i+2-2}.$$ Finally, the factor $\prod_{1\leq i \leq 2n-3}(1/i!) $ comes from the rest of the products involving the terms $\epsilon_i-\epsilon_j$.

\end{proof}

The proof of the main theorem of this section is similar to Theorem \ref{formula_limit}.
\begin{proof}[Proof of Theorem \ref{formula_limit_pfaffian}]
	(1) is clear from Lemma \ref{non_vanishing_ext_pfaffian} and Proposition \ref{finite_length_pfaffian}.

	(2) Apply the Selberg integral (Theorem \ref{Selberg_int_thm}) to the formula we got in Proposition \ref{existence_limit_pfaffian_j}. In this case we have $a=3, b=5, c=2$. Therefore we can further simplified (\ref{limit_int_pfaffian}) to
	\begin{align*}
		&\dfrac{(2n^2+n)!}{2n^2+n}\dfrac{1}{(n-1)!}\prod_{1\leq i \leq 2n}\dfrac{1}{i!}\prod^{n-2}_{i=0}\dfrac{(2+2i)!(2+2i)!(4+2i)!}{2(2n+3+2i)!}\\
		&=\dfrac{(2n^2+n)!}{2n^2+n}\dfrac{1}{(n-1)!}\prod_{1\leq i \leq n}\dfrac{1}{(2i)!}\dfrac{1}{(2i-1)!}\prod^{n-2}_{i=0}\dfrac{(2+2i)!(2+2i)!(4+2i)!}{2(2n+3+2i)!}\\
		&=\dfrac{(2n^2+n)!}{2n^2+n}\dfrac{1}{(n-1)!}\dfrac{1}{(2n)!}\prod_{0\leq i \leq n-2}\dfrac{1}{(2i+3)!}\prod^{n-2}_{i=0}\dfrac{(2+2i)!(4+2i)!}{2(2n+3+2i)!}\\
		&=(2n^2+n)!\dfrac{1}{(2n+1)!n!}\prod^{n-2}_{i=0}\dfrac{(2+2i)!(i+2)}{(2n+3+2i)!}\\
		&=(2n^2+n)!\dfrac{1}{(2n+1)!}\prod^{n-2}_{i=0}\dfrac{(2+2i)!}{(2n+3+2i)!}\\
		&=(2n^2+n)!\prod^{n-1}_{i=0}\dfrac{(2i)!}{(2n+1+2i)!}.
	\end{align*}
	and therefore by Proposition \ref{existence_limit_pfaffian_j} we have $$\frakj^{2n^2-n-1}(\Pf) = (2n^2+n)!\prod^{n-1}_{i=0} \dfrac{(2i)!}{(2n+1+2i)!},$$ which completes the proof of (2).

	(3) Applying Lemma \ref{ext_pfaffian} and Proposition \ref{injection_ext_pfaffian} then the proof is similar to Theorem \ref{formula_limit} (3).
\end{proof}


%

\section{Open questions}
Our approach relies on the vector space decompositions of $\Ext^j_S(S/I_n^d,S)$ and $\Ext^j_S(S/\Pf_{2n}^d,S)$ for $S=\Sym(\mathbb{C}^m \otimes \mathbb{C}^n)$ and $S=\Sym(\bigwedge^2 \mathbb{C}^n)$, respectively. When $S=\Sym(\Sym^2(\mathbb{C}^n))$, the corresponding decomposition of $\Ext$ module is still unknown, see also \cite[Remark 3.8]{Perlman} and \cite[Remark 2.7]{RaicuWeyman2}. However, we have seen the unexpected connection between $\epsilon^j(I_n)$ (resp. $\epsilon^j(\Pf_{2n})$) and the degree of Grassmannian (resp. Orthogonal Grassmannian). Therefore it is natural to ask the following questions:  

\begin{question}
	Let $S=\Sym(\Sym^2(\mathbb{C}^n))$. Let $\mathcal{S}_p$ be the ideal generated by $p \times p$ symmetric minors of $S$. 
	\begin{enumerate}
		\item For which $j$ does $\Ext^j_S(S/\mathcal{S}_{n-1}^D,S)$ have finite length?
		\item Suppose $\ell(\Ext^j_S(S/\mathcal{S}_{n-1}^D,S)) < \infty$ and nonzero. Do $\epsilon^j(\mathcal{S}_{n-1})$ and $\frakj^j(\mathcal{S}_{n-1})$ exist?
		\item Suppose $\epsilon^j(\mathcal{S}_{n-1})$ or $\frakj^j(\mathcal{S}_{n-1})$ exists. Can we identify it with the degree of Lagrangian Grassmannian?
	\end{enumerate}
\end{question}

Moreover, consider the results in \cite{JeffriesMontanoVarbaro}, we ask:
\begin{question}	
	\begin{enumerate}
		\item Are there any geometric interpretations for $\epsilon^0(I_p)$, $\epsilon^0(\Pf_{2p})$ and $\epsilon^0(\mathcal{S}_{p})$?
		\item Are there any geometric interpretations for $\frakj^0(I_p)$, $\frakj^0(\Pf_{2p})$ and $\frakj^0(\mathcal{S}_{p})$?
	\end{enumerate}
\end{question}

\section{Acknowledgement}
The author would like to thank her advisor Wenliang Zhang, for suggesting this problem and his guidance throughout the preparation of this paper, and Tian Wang for helpful suggestions.


\begin{thebibliography}{HKM}
\bibitem[Ap]{Apostol}
T.M.~Apostol, \emph{An Elementary View of Euler's Summation Formula},The American Mathematical Monthly, \textbf{106}, Taylor \& Francis, Ltd., 409--418, 1999. 

\bibitem[ADF]{AbeasisDelFra} S. ~Abeasis and A. ~Del Fra, \emph{Young diagrams and ideals of Pfaffians}, Adv. Math, \textbf{35}, (1980), 158 -- 178.

%

\bibitem [BR]{BuchsbaumRim}
D.~Buchsbaum, D.~Rim, \emph{A Generalized Koszul Complex. II. Depth and Multiplicity}, Trans. Amer. Math. Soc. ,\textbf{111}, (1963), 197--224.

\bibitem[BH]{BrunsHerzog}
W.~Bruns and J.~Herzog, \emph{Cohen-Macaulay rings}, Cambridge Studies in Advanced Mathematics, \textbf{39}. Cambridge University Press, Cambridge, 1993.

\bibitem[Cu]{Cutkosky}
S.D. ~Cutkosky, \emph{Asymptotic multiplicities of graded families of ideals and linear series},  Adv. Math. 264 (2014), 55–113. 

\bibitem[CHST]{CutkoskyHST}
S.D. ~Cutkosky, H.T. H\`{a}, H.~Srinivasan, E.~Theodorescu, \emph{Asymptotic Behavior of the Length of Local Cohomology}, Canadian Journal of Mathematics, \textbf{57} (2005), 1178 -- 1192.

\bibitem[DEP]{DeConciniEisenbudProcesi}
C. ~DeConcini, D. ~Eisenbud and C.Procesi, \emph{Young Diagrams and Determinantal Varieties}, Inventiones math \textbf{56} (1980), 129--165.


\bibitem[DM]{DaoMontano19}
H. ~Dao and J. ~Monta\~{n}o, \emph{Length of Local Cohomology of Powers of Ideals}, Trans.Amer.Soc,\textbf{371}. (2019), no 5,3483--3503.

\bibitem[DM2]{DaoMontano20}
H. ~Dao and J. ~Monta\~{n}o, \emph{On asymptotic vanishing behavior of local cohomology},  Math. Z. \textbf{295} (2020), no. 1-2, 73–86.

\bibitem[EH]{EisenbudHarris}
D. ~Eisenbud and J. ~Harris, \emph{3264 and all that -- A second course in algebraic geometry},  Cambridge University Press, Cambridge, 2016. xiv+616 pp. ISBN: 978-1-107-60272-4; 978-1-107-01708-5

\bibitem[FW]{ForresterWarnaar}
P.J.~Forrester and S.O.~Warnaar, \emph{The importance of Selberg integral}, Bull. Amer. Math. Soc, \textbf{45}, 2008, 489-534.

\bibitem[FH]{FultonHarris}
W.~Fulton and J.~Harris, \emph{Representation Theory
A First Course},  Graduate Texts in Mathematics, \textbf{129}. (2004), Springer, New York, NY.

\bibitem[Hu]{HunekeLecture}
C.~Huneke, \emph{Lectures on local cohomology}, Contemp. Math., \textbf{436}, Interactions between homotopy theory and algebra, 51--99, Amer. Math. Soc., Providence, RI, 2007. 

\bibitem[Hu81]{Huneke}
C.~Huneke, \emph{Powers of Ideals Generated
by Weak d-Sequences}, Journal of Algebra, \textbf{68}, 471--509 (1981).

\bibitem[JMV]{JeffriesMontanoVarbaro}
J.~Jeffries, J.~Monta\~{n}o and M.~Varbaro, \emph{Multiplicities of classical varieties}, Proc.London.Mac.Soc.(3), \textbf{110}, 1033--1055, 2015.

\bibitem[KV]{KatzValidashti}
D. ~Katz and J. ~Validashti, \emph{Multiplicities and Rees valuations}, Collect. Math, \textbf{61} (2010), 1-24.

\bibitem[Ke]{Kenkel}
J. ~Kenkel, \emph{Length of Local Cohomology of Thickenings}, arXiv:1912.02917.

\bibitem[Ra]{Raicu}
C. ~Raicu, \emph{Regularity and cohomology of determinantal thickenings}, Proc. Lond. Math. Soc. (3) \textbf{116} (2018), no. 2, 248--280.

\bibitem[RW14]{RaicuWeyman}
C. ~Raicu and J. ~Weyman, \emph{
Local cohomology with support in generic determinantal ideals}, Algebra and Number Theory. \textbf{8}, 2014, no. 5, 1231--1257.

\bibitem[RW16]{RaicuWeyman2}
C. ~Raicu and J. ~Weyman, \emph{Local cohomology with support in ideals of symmetric minors and Pfaffians.}, J. Lond. Math. Soc. (2) \textbf{94} (2016), no. 3, 709–725. 

\bibitem[RWW]{RaicuWeymanWitt}
C.~Raicu, J.~Weyman and E.~Witt, \emph{Local cohomology with support in ideals of maximal minors and sub-maximal Pfaffians}, Advances in Mathematics. \textbf{250}, 2014, 596-610.

\bibitem[Pe]{Perlman}
M. ~Perlman, \emph{Regularity and cohomology of Pfaffian thickenings}, J. Commut. Algebra, \textbf{13}, 2021, 523 -- 548.

\bibitem[Se]{Selberg}
A. ~Selberg, \emph{Bemerkninger om et multipelt integral}, Norsk. Math. Tidsskr, \textbf{24}, 1944, 71-78.

\bibitem[To]{Totaro}
B. ~Totaro, \emph{Towards a Schubert calculus for complex reflection groups},  Math. Proc. Cambridge Philos. Soc. \textbf{134} (2003), no. 1, 83–93.

\bibitem[UV]{UlrichValidashti}
B. ~Ulrich and J. ~Validashti, \emph{Numerical criteria for integral dependence},  Math. Proc. Cambridge Philos. Soc. \textbf{151} (2011), no. 1, 95–102.

\bibitem[We]{Weyman}
J.~Weyman, \emph{Cohomology of Vector Bundles and Syzygies}, Cambridge University Press,Cambridge,2003.

\end{thebibliography}
\end{document}